\pdfoutput=1
\RequirePackage{ifpdf}
\ifpdf 
\documentclass[pdftex]{sigma}
\else
\documentclass{sigma}
\fi

\numberwithin{equation}{section}

\newtheorem{Theorem}{Theorem}[section]
\newtheorem{Corollary}[Theorem]{Corollary}
\newtheorem{Lemma}[Theorem]{Lemma}
 { \theoremstyle{definition}
\newtheorem{Remark}[Theorem]{Remark} }

\def\Xint#1{\mathchoice
{\XXint\displaystyle\textstyle{#1}}%
{\XXint\textstyle\scriptstyle{#1}}%
{\XXint\scriptstyle\scriptscriptstyle{#1}}%
{\XXint\scriptscriptstyle\scriptscriptstyle{#1}}%
\!\int}
\def\XXint#1#2#3{{\setbox0=\hbox{$#1{#2#3}{\int}$}
\vcenter{\hbox{$#2#3$}}\kern-.5\wd0}}

\def\dashint{\Xint-}

\newcommand{\intZ}{\mathbb{Z}}
\newcommand{\prodprime}{\sideset{}{'}\prod}
\newcommand{\lHopital}{l'H\^{o}pital}
\newcommand{\Weylchamber}{\mathbb{W}}
\newcommand{\bigO}{\mathcal{O}}

\DeclareMathOperator*{\Res}{Res}
\DeclareMathOperator{\id}{id}

\begin{document}


\renewcommand{\thefootnote}{$\star$}

\newcommand{\arXivNumber}{1512.01612}

\renewcommand{\PaperNumber}{037}

\FirstPageHeading

\ShortArticleName{The Transition Probability of the $q$-TAZRP ($q$-Bosons) with Inhomogeneous Jump Rates}

\ArticleName{The Transition Probability of the $\boldsymbol{q}$-TAZRP\\ ($\boldsymbol{q}$-Bosons) with Inhomogeneous Jump Rates\footnote{This paper is a~contribution to the Special Issue
on Asymptotics and Universality in Random Matrices, Random Growth Processes, Integrable Systems and Statistical Physics in honor of Percy Deift and Craig Tracy.
The full collection is available at \href{http://www.emis.de/journals/SIGMA/Deift-Tracy.html}{http://www.emis.de/journals/SIGMA/Deift-Tracy.html}}}

\Author{Dong WANG~$^\dag$ and David WAUGH~$^\ddag$}

\AuthorNameForHeading{D.~Wang and D.~Waugh}

\Address{$^\dag$~Department of Mathematics, National University of Singapore, 119076, Singapore}
\EmailD{\href{mailto:matwd@nus.edu.sg}{matwd@nus.edu.sg}}
\URLaddressD{\url{http://www.math.nus.edu.sg/~matwd/}}

\Address{$^\ddag$~BNP Paribas, 787 7th Avenue, New York, NY, 10019, USA}
\EmailD{\href{mailto:david.waugh@u.nus.edu}{david.waugh@u.nus.edu}}

\ArticleDates{Received January 03, 2016, in f\/inal form April 08, 2016; Published online April 14, 2016}

\Abstract{In this paper we consider the $q$-deformed totally asymmetric zero range process ($q$-TAZRP), also known as the $q$-boson (stochastic) particle system, on the $\intZ$ lattice, such that the jump rate of a particle depends on the site where it is on the lattice. We derive the transition probability for an $n$ particle process in Bethe ansatz form as a sum of $n!$ $n$-fold contour integrals. Our result generalizes the transition probability formula by Korhonen and Lee for $q$-TAZRP with a homogeneous lattice, and our method follows the same approach as theirs.}


\Classification{82B23; 82C22; 60J35}

\renewcommand{\thefootnote}{\arabic{footnote}}
\setcounter{footnote}{0}

\section{Introduction}

The zero range process (ZRP) is a well studied model in non-equilibrium stochastic mechanics. Since it was introduced by Spitzer \cite{Spitzer70}, it has become a popular model in both the physics and mathematics communities \cite{Balazs-Komjathy08}.

With properly chosen parameters, as in the $q$-TAZRP to be described below, the $1$-di\-men\-sio\-nal ZRP model is integrable, and is solvable by the Bethe ansatz \cite{Borodin-Corwin-Sasamoto14,Povolotsky13, Sasamoto-Wadati98}. The application of the Bethe ansatz to f\/ind the explicit transition pro\-ba\-bilities in $1$-dimensional particle models has come along two decades since \cite{Gwa-Spohn92} and \cite{Schutz97}, and it has yielded the explicit transition probability formulas for the asymmetric simple exclusion process (ASEP) in Tracy and Widom's seminal work~\cite{Tracy-Widom08}. These exact formulas enabled (after further manipulation) Tracy and Widom to calculate the asymptotic distribution of one particle as the number of particles approaches inf\/inity with special initial conditions \cite{Tracy-Widom09}, and greatly improve our understanding of the KPZ universality~\cite{Corwin11,Quastel12}. Since then, dif\/ferent methods have been developed and exact formulas for the transition probabilities have been studied for various integrable $1$-dimensional particle models~\cite{Lee11}, including the $q$-deformed totally asymmetric simple exclusion process ($q$-TASEP), which is an example of the general framework of Macdonald processes~\cite{Borodin-Corwin13} and has duality with $q$-TAZRP~\cite{Borodin-Corwin-Sasamoto14}.

In our paper, we consider the $q$-deformed totally asymmetric zero range process ($q$-TAZRP), of which a def\/inition is given below. This model was considered by Borodin and Corwin as a~specialization of Macdonald processes, and then by Borodin, Corwin and Sasamoto in \cite{Borodin-Corwin-Sasamoto14} in the light of its stochastic duality (in the sense of \cite[Section~2.3]{Liggett05}) to the $q$-TASEP model. Our main inspiration is the paper~\cite{Korhonen-Lee14} by Korhonen and Lee, where the transition probability for the $q$-TAZRP with homogeneous jump rates is derived. In~\cite{Korhonen-Lee14} the authors used Tracy and Widom's method rather than results from Macdonald processes. Our paper is a generalization to~\cite{Korhonen-Lee14}, and we extend their method to inhomogeneous jump rates. Let us note that, this integrable ZRP was f\/irst introduced by Sasamoto and Wadati~\cite{Sasamoto-Wadati98} under the name of $q$-Boson totally asymmetric dif\/fusion model, and this name is used in later publications like~\cite{Borodin-Corwin-Petrov-Sasamoto15}, so we refer to this name in the title. But later in this paper we follow the nomenclature in~\cite{Korhonen-Lee14} and call the model $q$-TAZRP.

As mentioned above, $q$-TAZRP and $q$-TASEP have the stochastic duality. There is a natural particle-spacing duality between zero range process and simple exclusion process (SEP) models generally \cite{Spitzer70}, such that if particles $\{ x_i(t) \}$ are in a simple exclusion process, then their spacings $\{ y_i = x_i - x_{i + 1} \}$ are in a corresponding zero range process. While in~\cite{Borodin-Corwin-Sasamoto14}, a dif\/ferent stochastic duality that is associated to a duality function is observed, and it is applied there to derive moment formulas of $q$-TASEP. Therefore it is natural that the exact formulas enjoyed by $q$-TASEP have a dual version for $q$-TAZRP, and vice versa, although the relation between the formulas may not be direct. However, there is a certain discrepancy between the results for $q$-TASEP and those for $q$-TAZRP. The $q$-TASEP as a specialization of Macdonald process naturally contains parameters, which are independent of $q$, for the jumping rates of particles, so the particles may not be identical, but some are fast particles and some are slow. In the study of $q$-TAZRP, Korhonen and Lee assume that the jump rate of a particle only depends on the height of the stack where the particle belongs, and the model has no parameter other than $q$.

The particle-spacing duality between ZRP and SEP models suggests if a SEP model can have fast and slow particles and still be integrable, then the dual ZRP model can both be integrable and have inhomogeneous jump rates, that is, some lattice sites are easy to pass through and some are dif\/f\/icult. The purpose of our paper is to generalize the result in \cite{Korhonen-Lee14} to the $q$-TAZRP model with inhomogeneous jump rates. In this sense, we remove the discrepancy between the existing results for $q$-TASEP and $q$-TAZRP.

Before going into detail, we remark that the $q$-TAZRP with homogeneous jump rates can be analysed by spectral theory \cite{Borodin-Corwin-Petrov-Sasamoto15}, and proving the transition probability formula is equivalent to proving the direct Plancherel theory, see \cite[Section 7.5]{Borodin-Corwin-Petrov-Sasamoto15a} for an explanation in the ASEP case.
It would be very interesting to see if our transition probability results likewise lead to a direct Plancherel theory for the inhomogeneous rate model~-- though such a theory has not yet been developed.

It is worth noting that very recently Borodin and Petrov investigated the fully inhomogeneous stochastic higher spin six vertex model in a quadrant \cite{Borodin-Petrov16}, and the spectral theory was analysed. Our $q$-TAZRP model, together with the above mentioned $q$-TASEP and ASEP models, can all be regarded as degenerations of the general higher spin six vertex model.

Below we give a description of the $q$-TAZRP model. We keep the discussion on the background to a minimum, since \cite{Korhonen-Lee14} has a detailed introduction to this ef\/fect.

\subsection[Description of $q$-TAZRP]{Description of $\boldsymbol{q}$-TAZRP}

Throughout this paper, $q \in (0, 1)$ is a constant. Let $n$ particles be on the integer lattice $\intZ$, and denote the position of the $k$-th particle by $x_k$, or $x_k(t)$ if we specify the time $t \geq 0$. We also use $x_k$ to denote the $k$-th particle itself by abuse of notation. We assume that initially $(x_1(0), x_2(0), \dotsc, x_n(0)) =: X(0) \in \Weylchamber^n$, where $\Weylchamber^n = \{ (i_1, \dotsc, i_n) \in \intZ^n \,|\, i_1 \geq i_2 \geq \dotsb \geq i_n \}$. From the dynamics given below, we know that for all $t > 0$ the order $x_1(t) \geq x_2(t) \geq \dotsb \geq x_n(t)$ is kept, or equivalently, $X(t) \in \Weylchamber^n$.

Each particle is at an integer site at any given time, and several consecutive particles can share a site. Suppose $x_{k - 1}(t) > x_k(t) = x_{k + 1}(t) = \dotsb = x_{k + j - 1}(t) = m > x_{k + j}(t)$, then we say that at time $t$, the particles $x_k, x_{k + 1}, \dotsc, x_{k + j - 1}$ are in a stack at $m$ with height $j$, such that $x_k$ is at the top of the stack, $x_{k + i}$ is below $x_{k + i - 1}$ ($i = 1, \dotsc, j - 1$), and $x_{k + j - 1}$ is at the bottom. See Fig.~\ref{fig:particles}.
\begin{figure}[htb]
 \begin{minipage}[t]{0.46\linewidth}
 \centering
 \includegraphics{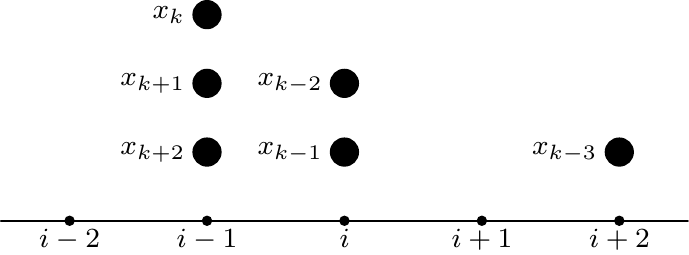}
 \caption{Particles in $q$-TAZRP. The $k$-th particle is on the top of a stack at $i - 1$ with height~$3$.}
 \label{fig:particles}
 \end{minipage}
 \hfill
 \begin{minipage}[t]{0.48\linewidth}
 \centering
 \includegraphics{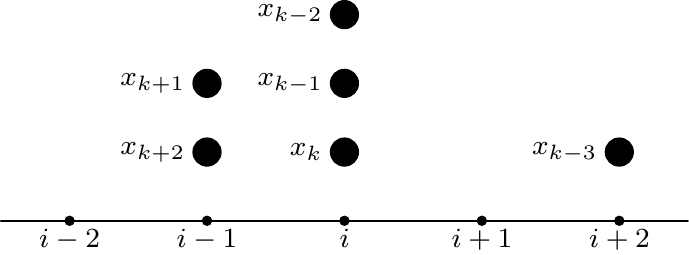}
 \caption{The $k$-th particle jumps to the bottom of the stack at $i$. Until the $k$-th particle jumps, the $(k + 1)$-th and $(k + 2)$-th particles stay put.}
 \label{fig:particles_jumped}
 \end{minipage}
\end{figure}

On each site $i$, we assign a \emph{conductance} $a_i > 0$. Throughout this paper, we denote $b_i = a_i(1 - q)$ for notational simplicity. The dynamics of the particles is given as follows.
If a particle is in a stack but is not on the top of the stack, then it cannot move; otherwise if the particle is on the top of a stack at $m$ with height $j$, it moves to site $m + 1$ at rate $a_m(1 - q^j) = b_m(1 + q + \dotsb + q^{j - 1})$.
The movement occurs regardless of whether the right neighbour site is occupied or not (hence the model is named ``zero range''). If the right neighbour site is occupied, then $x_k$ moves to the bottom of the stack there, and the height of that stack increases by $1$. See Fig.~\ref{fig:particles_jumped} for illustration.

The $q$-TAZRP is clearly a Markov process. Let $X = (x_1, x_2, \dotsc, x_n) \in \Weylchamber^n$ denote a state of the model, then $P(X; t)$, the probability that the model is at state $X$ at time $t$, satisf\/ies the master equation
\begin{gather} \label{eq:abstract_master_eq}
 \frac{d}{dt} P(X; t) = H P(X; t) = \sum_{Y \in \Weylchamber^n} P(Y; t) (c(Y, X) - c(X, Y)),
\end{gather}
where $H$ is the generator of the Markov process, and $c(X, Y)$ is the transition rate for the Markov process from state $X$ to state $Y$. To express the master equation more concretely, for $X = (x_1, \dotsc, x_n) \in \Weylchamber^n$ we denote
\begin{gather*}
 X^{k, -} = (x_1, \dotsc, x_{k - 1}, x_k - 1, x_{k + 1}, \dotsc, x_n), \qquad k = 1, \dotsc, n,
\end{gather*}
and denote
\begin{gather} \label{eq:various_notations}
 \begin{split}
 & \ell(X) :=  \text{$\#$ of the distinct values attained by $x_1, \dotsc, x_n$}, \\
 & v_i(X) :=  \text{the $i$-th largest value attained by $x_1, \dotsc, x_n$}, \\
& p_k(X) :=  \text{$\#$ of $x_1, \dotsc, x_n$ whose value is $k$}, \\
 & n_i(X) :=  p_{v_i}(X), \quad \text{and} \quad N_i(X) := n_1(X) + \dotsb + n_i(X).
 \end{split}
\end{gather}
For instance, if
\begin{gather} \label{eq:general_X}
 X = (\underbrace{s_1, \dotsc, s_1}_{k_1}, \underbrace{s_2, \dotsc, s_2}_{k_2}, \dotsc, \underbrace{s_m, \dotsc, s_m}_{k_m}), \qquad s_1 > s_2 > \dotsb > s_m,
\end{gather}
and let $K_i = k_1 + k_2 + \dotsb + k_i$, then $\ell(X) = m$, $v_i(X) = s_i$, $n_i(X) = k_i$, and $N_i(X) = K_i$. For~$X$ expressed in~\eqref{eq:general_X}, the master equation~\eqref{eq:abstract_master_eq} can be expressed concretely as
\begin{gather} \label{eq:master_eq_concrete}
 \frac{d}{dt} P(X; t) = \sum^m_{i = 1} \big(1 - q^{p_{s_i - 1}(X^{K_i, -})}\big) a_{s_i - 1} P\big(X^{K_i, -}; t\big) - \sum^m_{i = 1} \big(1 - q^{k_i}\big) a_{s_i} P(X; t),
\end{gather}
where $p_{s_i - 1}(X^{K_i, -})$ is the number of particles on site $s_i - 1$ for the state $X^{K_i, -}$ (as def\/ined in~\eqref{eq:various_notations}). Here we remark that if $X \in \Weylchamber^n$ and $k \in \{ 1, \dotsc, n \}$, $X^{k, -}$ may not be in $\Weylchamber^n$. But all~$X^{K_i, -}$ on the right-hand side of~\eqref{eq:master_eq_concrete} are in~$\Weylchamber^n$.

\subsection{Main result}

In this paper we compute $P_Y(X; t)$, the probability that the particles are in state $X \!= \! (x_1, \dotsc, x_n)$ at time $t > 0$, given the initial condition that the particles are in state $Y = (y_1, \dotsc, y_n)$ at time~$0$. Equivalently, $P_Y(X; t)$ is the solution to the master equation \eqref{eq:abstract_master_eq} with initial condition $P_Y(X; 0) = 1$ if $X = Y$ and $P_Y(X; t) = 0$ if $X \in \Weylchamber^n {\setminus} \{ Y \}$.

To state the main theorem, we introduce some notations. First, for~$X \in \Weylchamber^n$ in the form of~\eqref{eq:general_X}, denote
\begin{gather*}
 W(X) = \prod^{\ell(X)}_{i = 1} \prod^{n_i(X)}_{j = 1} \frac{1 - q^j}{1 - q} = \prod^m_{i = 1} [k_i]_q!,
\end{gather*}
where we use the $q$-deformed factorial $[k]_q! = (1 - q)(1 - q^2) \dotsb (1 - q^k)/(1 - q)^k$. Next, let $w_1, \dotsc, w_n$ be variables. For each pair of integers $\alpha < \beta$, denote
\begin{gather*}
 S_{(\beta, \alpha)} = -\frac{q w_{\beta} - w_{\alpha}}{q w_{\alpha} - w_{\beta}}.
\end{gather*}
For a permutation $\sigma \in S_n$, we say $(\beta, \alpha)$ is an \emph{inversion} of $\sigma$ if $\alpha < \beta$ and $\sigma^{-1}(\alpha) > \sigma^{-1}(\beta)$. Then let
\begin{gather*}
 A_{\sigma}(w_1, \dotsc, w_n) = \prod_{(\beta, \alpha) \text{ is an inversion of } \sigma} S_{(\beta, \alpha)}.
\end{gather*}
We write $A_{\sigma} = A_{\sigma}(w_1, \dotsc, w_n)$ if there is no possibility of confusion, but later we occasionally use dif\/ferent variables for $A_{\sigma}$. Note that for $\sigma = \id \in S_n$, since the identity permutation has no inversion, we let $A_{\id} = 1$. For $X = (x_1, \dotsc, x_n) \in \intZ^n$, $Y = (y_1, \dotsc, y_n) \in \intZ^n$ and $\sigma \in S_n$, def\/ine
\begin{gather} \label{eq:integral_Lambda}
 \Lambda_Y(X; t; \sigma) = \left( \prod^n_{k = 1} \frac{-1}{b_{x_k}} \right) \dashint_{C_1} dw_1 \dotsi \dashint_{C_n} dw_n A_{\sigma} \prod^n_{j = 1} \left[ \prodprime^{x_j}_{k = y_{\sigma(j)}} \left( \frac{b_k}{b_k - w_{\sigma(j)}} \right) e^{-w_j t} \right],
\end{gather}
where the notation $\dashint_{C_i} dw_i$ is a shorthand for $(2\pi i)^{-1} \oint_{C_i} dw_i$ and the notation $\prod'$ is an extension of the usual $\prod$ notation such that
\begin{gather*}
 \prodprime^n_{k = m} f(k) =
 \begin{cases}
 \displaystyle \prod^n_{k = m} f(k) & \text{if $n \geq m$}, \\
 1 & \text{if $n = m - 1$}, \\
\displaystyle \prod^{m - 1}_{k = n + 1} \frac{1}{f(k)} & \text{if $n < m - 1$},
 \end{cases}
\end{gather*}
for example, 
 $\prod\limits^{-3}_{k = 0}\!\!{\vphantom{\prod}}' f(k) = 1/[f(-1) f(-2)]$. Later we use the identity
\begin{gather*}
 \left( \prodprime^x_{k = m + 1} f(k) \right) \left( \prodprime^m_{k = y} f(k) \right) = \prodprime^x_{k = y} f(k).
\end{gather*}
With respect to any $w_j$, the integrand in \eqref{eq:integral_Lambda} has explicit poles in the form of $w_j = b_k$ for some~$k$, and the factor $A_{\sigma}$ may introduce poles in the form of $w_j = q w_i$ and $w_j = q^{-1} w_l$ for some~$i$ and~$l$ such that $i < j$ and $l > j$. We require that poles $b_k$ and $q w_i$ are enclosed by~$C_j$, but~$q^{-1} w_l$ are not. To be concrete, we can take that $C_j = \{ \lvert z \rvert = R \}$ for all~$j$ with a large enough~$R$, and choose the counter-clockwise orientation.
\begin{Theorem} \label{thm:main}
 Let $X, Y \in \Weylchamber^n$. Then
 \begin{gather} \label{eq:main_result}
 P_Y(X; t) = \frac{1}{W(X)} \sum_{\sigma \in S_n} \Lambda_Y(X; t; \sigma).
 \end{gather}
\end{Theorem}

\begin{Remark}
 The prefactor $W(X)$ is the same $W(X)$ in \cite[Theorem~2.1]{Korhonen-Lee14} in the homogeneous case. It corresponds to the fact that the model is PT symmetric with respect to the invariant measure, and $W(X)$, like the $C_q(\vec{n})$ in~\cite{Borodin-Corwin-Petrov-Sasamoto15}, gives the invariant measure for the $q$-TAZRP process, up to a constant. See \cite[Remark~2.3]{Borodin-Corwin-Petrov-Sasamoto15} and \cite[Corollary~3.3.12]{Borodin-Corwin13}.
\end{Remark}
With special initial conditions, the transition probability formula~\eqref{eq:main_result} can be simplif\/ied. Here we consider the initial condition $y_1 = y_2 = \dotsb = y_n = 0$. To state the result, we denote
\begin{gather}
 B(w_1, \dotsc, w_n) = \prod_{1 \leq i < j \leq n} \frac{w_i - w_j}{w_i - q w_j}.
\end{gather}
\begin{Corollary} \label{cor:step_init}
 Let $X \in \Weylchamber^n$. Then
 \begin{gather}
 P_{(0, 0, \dotsc, 0)}(X; t) = [n]_q! \frac{1}{W(X)} \left( \prod^n_{k = 1} \frac{-1}{b_{x_k}} \right) \nonumber\\
 \hphantom{P_{(0, 0, \dotsc, 0)}(X; t) =}{}
 \times \dashint_{C_1} dw_1 \dotsi \dashint_{C_n} dw_n B(w_1, \dotsc, w_n) \prod^n_{j = 1} \left[ \prodprime^{x_j}_{k_j = 0} \left( \frac{b_{k_j}}{b_{k_j} - w_j} \right) e^{-w_j t} \right].\label{eq:step_init_trans_prob}
 \end{gather}
\end{Corollary}

We note that the transition probability formula \eqref{eq:step_init_trans_prob} shares formal similarity with the moment formulas in the $q$-TASEP with various initial conditions, see \cite[Theorem~2.11 and Corollary~2.12]{Borodin-Corwin-Sasamoto14}, where the parameters~$a_i$ correspond to the~$a_i$ in our paper. This is not a coincidence: The duality between the $q$-TASEP and $q$-TAZRP implies that the moment formula for $q$-TASEP satisf\/ies the true evolution equation for $q$-TAZRP, as explained in~\cite{Borodin-Corwin-Sasamoto14}. In the special case where $a_i$ are identical, moment formulas in $q$-TASEP are derived from properties of $q$-TAZRP in~\cite{Borodin-Corwin-Petrov-Sasamoto15}.

\subsection*{Organization of the paper}

The basic idea of our proof to Theorem \ref{thm:main} is the Bethe ansatz, following the approach of~\cite{Tracy-Widom08} and~\cite{Korhonen-Lee14}. In Section~\ref{sec:1pt} we consider the one particle case where the interaction between particles is missing as a warm-up. In Section~\ref{sec:2_particle_case} the two particle case, the simplest non-trivial case, is analysed explicitly, and in Section~\ref{sec:n_particle_case} the general case is worked out. The essence of the Bethe ansatz is that the interaction among multiple particles is equivalent to the superposition of two-particle interactions. We use the idea implicitly. At last, we prove Corollary~\ref{cor:step_init} in Section~\ref{sec:proof_of_cor}.

\section{One particle case} \label{sec:1pt}

In the simplest case that there are only one particle $x_1$, the master equation \eqref{eq:master_eq_concrete} becomes (we write $X = (x)$ as $x$)
\begin{gather*} 
 \frac{d}{dt} P(x; t) = b_{x - 1} P(x - 1; t) - b_{x} P(x; t).
\end{gather*}
Given the initial condition that the particle is at $y \in \intZ$ at time $t = 0$, that is, $P(y; 0) = 1$ and $P(x; 0) = 0$ if $x \neq y$, we can compute $P_y(x; t)$ by elementary probability. Since the particle jumps only to the right, so $P_y(x; t) = 0$ for all $x < y$. The time that the particle f\/irst jumps away from $y$ is in the exponential distribution with rate~$b_y$, and after it jump away, it cannot be back, so $P_y(y; t) = e^{-b_y t}$. Next, $P_y(y + 1; t)$ is the probability that the particle jumps exactly once from time~$0$ to~$t$, and it is
\begin{gather*}
 \int^t_0 b_y e^{-b_y s} e^{-b_{y + 1} (t - s)} ds = b_y (b_{y + 1} - b_y)^{-1} \big(e^{-b_y t} - e^{-b_{y + 1} t}\big), \qquad \text{if $b_{y + 1} \neq b_y$}.
\end{gather*}
Inductively, we use
\begin{gather*}
 P_y(y + k; t) =   \int^t_0 b_y e^{-b_y s_1} \left( \int^t_{s_1} b_{y + 1} e^{-b_{y + 1} s_2} \left( \int^t_{s_2} \dotsb \vphantom{\int^t_{s_{k - 1}}} \right. \right. \\
\hphantom{P_y(y + k; t) =}{} \times \left. \left. \left( \int^t_{s_{k - 1}} b_{y + k - 1} e^{-b_{y + k - 1} s_k} e^{-b_{y + k} (t - s_k)} ds_k \right) \dotsb ds_3 \right) ds_2 \right) ds_1 \\
\hphantom{P_y(y + k; t)}{} =  \int^t_0 b_y e^{-b_y s_1} P_{y + 1}(y + k; t - s_1) ds_1,
\end{gather*}
and derive
\begin{gather*}
 P_y(x; t) =
 \begin{cases}
 0 & \text{if $x < y$}, \\
 e^{-b_x t} & \text{if $x = y$}, \\
 \displaystyle \left( \prod^{x - 1}_{k = y} b_k \right) \sum^x_{k = y} \left( \prod_{y \leq j \leq x,\ j \neq k} \frac{1}{b_j - b_k} \right) e^{-b_k t} & \text{if $x > y$}.
 \end{cases}
\end{gather*}
Here and later, we assume $b_k$ are distinct, and understand the formulas by \lHopital's rule if some of them are identical. A more convenient way to express $P_y(x; t)$ is ($w = 1 - z^{-1}$)
\begin{subequations}
 \begin{gather}
 P_y(x; t) =   \frac{-1}{b_x} \dashint_C \prodprime^x_{k = y} \frac{b_k}{b_k - w} e^{-wt} dw \label{eq:1_particle_a} \\
\hphantom{P_y(x; t)}{}  =   \frac{1}{b_x} \dashint_{\tilde{C}} \left( \prodprime^x_{k = y} \frac{b_k z}{1 - z(1 - b_k)} \right) e^{(z^{-1} - 1)t} \frac{dz}{z^2}, \label{eq:1_particle_b}
 \end{gather}
\end{subequations}
where the contour $C$ encloses all the poles $b_k$ in positive orientation, and the contour $\tilde{C}$ encloses the pole $0$ in positive orientation. If some $b_k \neq 1$, then $\tilde{C}$ should not enclose poles $(1 - b_k)^{-1}$. Thus the $n = 1$ case of Theorem~\ref{thm:main} is proved.

\begin{Remark}
 Although the expression~\eqref{eq:1_particle_b} of $P_y(x; t)$ seems unnatural, it reduces to the $n = 1$ case \cite[Theorem~2.1]{Korhonen-Lee14} when $b_k = 1$ for all $k \in \intZ$.
\end{Remark}

\section{Two particle case} \label{sec:2_particle_case}

In the $2$-particle case, where the two particles are at $x_1 \geq x_2$, the master equation \eqref{eq:master_eq_concrete} can be written as
\begin{gather} \label{eq:master_eq_2_particles}
 \frac{d}{dt} P((x_1, x_2); t) =
 \begin{cases}
 b_{x_1 - 1} P((x_1 - 1, x_2); t) - b_{x_1} P((x_1, x_2); t) & \\
{} + b_{x_2 - 1} P((x_1, x_2 - 1); t) - b_{x_2} P((x_1, x_2); t),
 & \text{if $x_1 > x_2 + 1$}, \\
 (1 + q) b_xP((x, x); t) - b_{x + 1} P((x + 1, x); t) & \text{if $x_2 = x$} \\
{} + b_{x - 1} P((x + 1, x - 1); t) - b_x P((x + 1, x); t),
& \text{and $x_1 = x + 1$}, \\
 b_{x - 1} P((x, x - 1); t) - (1 + q) b_x P((x, x); t) & \text{if $x_1 = x_2 = x$}.
 \end{cases}\!\!\!\!
\end{gather}
and we want to solve it under the initial condition ($y_1 \geq y_2$)
\begin{gather} \label{eq:init_C_2_particles}
 P((y_1, y_2); 0) = 1, \qquad \text{and} \qquad P((x_1, x_2); 0) = 0 \qquad \text{for all $(x_1, x_2) \in \Weylchamber^2 {\setminus} \{ (y_1, y_2) \}$}.\!\!\!
\end{gather}
To simplify the master equation, we consider functions $u^0((x_1, x_2); t)$ that are dif\/ferentiable in $t \geq 0$, for all $(x_1, x_2) \in \intZ^2$. By contrast, the probabilities $P((x_1, x_2); t)$ can also be viewed as a~dif\/ferentiable function in $t$, but it is only def\/ined when $(x_1, x_2) \in \Weylchamber^2$.

If the functions $u^0((x_1, x_2); t)$ satisfy the dif\/ferential equations (so-called free particle evolution equations)
\begin{gather}
 \frac{d}{dt} u^0((x_1, x_2); t) = b_{x_1 - 1} u^0((x_1 - 1, x_2); t) - b_{x_1} u^0((x_1, x_2); t) \nonumber\\
 \hphantom{\frac{d}{dt} u^0((x_1, x_2); t) =}{} + b_{x_2 - 1} u^0((x_1, x_2 - 1); t) - b_{x_2} u^0((x_1, x_2); t),\label{eq:evolution_eq_2_particles}
\end{gather}
and in addition, they satisfy the relations (so-called boundary condition)
\begin{gather} \label{eq:BC_2_particles}
 b_{x - 1} u^0((x - 1, x); t) = b_{x - 1} q u^0((x, x - 1); t) + b_x (1 - q) u^0((x, x); t) \qquad \text{for all $x \in \intZ$},\!\!\!
\end{gather}
then we have that the functions with $t \geq 0$ and $(x_1, x_2) \in \Weylchamber^2$
\begin{gather} \label{eq:relation_u^0_u_2_particles}
 u((x_1, x_2); t) = \frac{1}{W((x_1, x_2))} u^0((x_1, x_2); t),
 \end{gather}
where
\begin{gather*}
W((x_1, x_2)) =
 \begin{cases}
 1 & \text{if $x_1 > x_2$}, \\
 1 + q & \text{if $x_1 = x_2$},
 \end{cases}
\end{gather*}
satisfy the master equation \eqref{eq:master_eq_2_particles} with $P((x_1, x_2); t) = u((x_1, x_2); t)$. To see it, we note that if $x_1 > x_2 + 1$, then \eqref{eq:master_eq_2_particles} is the same as \eqref{eq:evolution_eq_2_particles} with $P((x_1, x_2); t) = u^0((x_1, x_2); t)$; if $x_1 = x + 1$ and $x_2 = x$, then~\eqref{eq:master_eq_2_particles} is equivalent to~\eqref{eq:evolution_eq_2_particles} with $P((x + 1, x); t) = u^0((x + 1, x); t)$, $P(x + 1, x - 1); t) = u^0((x + 1, x - 1); t)$, and $P((x, x); t) = (1 + q)^{-1} u^0((x_1, x_2); t)$; if $x_1 = x_2 = x$, and we let $P((x, x); t) = u^0((x, ); t)$ and $P((x, x - 1); t) = u^0((x, x - 1); t)$, then the right-hand side of then~\eqref{eq:master_eq_2_particles} is the same as $(1 + q)^{-1}$ times \eqref{eq:evolution_eq_2_particles}, assuming the identity~\eqref{eq:BC_2_particles}. Thus~\eqref{eq:evolution_eq_2_particles},~\eqref{eq:BC_2_particles} and~\eqref{eq:relation_u^0_u_2_particles} imply that $u((x_1, x_2); t)$ satisfy~\eqref{eq:master_eq_2_particles}.

Hence by solving \eqref{eq:evolution_eq_2_particles} with the boundary condition \eqref{eq:BC_2_particles} and the partial initial condition on $\Weylchamber^2$: $u^{(0)}(y_1, y_2; 0) = W((y_1, y_2))$ and $u^{(0)}(x_1, x_2; 0) = 0$ if $(x_1, x_2) \in \Weylchamber^2 {\setminus} \{ (y_1, y_2) \}$, we solve the master equation~\eqref{eq:master_eq_2_particles} with the initial condition~\eqref{eq:init_C_2_particles}.

\begin{Remark} \label{rmk:inverse_2pt}
 It is not obvious that from a solution $u((x_1, x_2); t)$ to \eqref{eq:master_eq_2_particles} where $(x_1, x_2) \in \Weylchamber^2$, we can construct $u^0((x_1, x_2); t)$ for $(x_1, x_2) \in \intZ^2$, that satisfy \eqref{eq:relation_u^0_u_2_particles} if $(x_1, x_2) \in \Weylchamber^2$, as well as solve \eqref{eq:evolution_eq_2_particles} and \eqref{eq:BC_2_particles}. Although it is a digression of our main theorem, for its own interest, we describe the construction is as follows:
\begin{gather*}
 u^0((x_1, x_2); t) = u((x_1, x_2); t) \quad \text{if $x_1 > x_2$}, \qquad u^0((x, x); t) = (1 + q) u((x, x); t), \\
 u^0((x - 1, x); t) = q u((x, x - 1); t) + \frac{b_x}{b_{x - 1}} \big(1 - q^2\big) u((x, x); t),
\end{gather*}
and then inductively with respect to $k \geq 1$, under the assumption that any $u^0((x_1, x_2); t)$ is expressed as linear combinations of $u((z_1, z_2); t)$ if $x_1 \geq x_2 - k$,
\begin{gather}
 u^0((x - (k + 1), x); t) =
 \frac{1}{b_{x - (k + 1)}} \bigg[ \frac{d}{dt} u^0((x - k, x); t) \nonumber\\
 \hphantom{u^0((x - (k + 1), x); t) = }{} - b_{x - 1} u^0((x - k, x - 1); t) + (b_{x - k} + b_x) u^0((x - k, x); t) \bigg].\label{eq:indct_u^0_in_u}
\end{gather}
Note that in \eqref{eq:indct_u^0_in_u}, not only are $u^0((x - k, x - 1); t)$ and $u^0((x - k, x); t)$ expressed by linear combinations of $u((z_1, z_2); t)$, but also $\frac{d}{dt} u^0((x - k, x); t)$ is the linear combinations of $\frac{d}{dt} u((z_1, z_2); t)$, and in turn the linear combinations of $u((z_1, z_2); t)$. For example, for $k = 1$,
\begin{gather*}
   u^0((x - 2, x); t) \\
 =   \frac{1}{b_{x - 2}} \left[ \frac{d}{dt} u^0((x - 1, x); t) - b_{x - 1} u^0((x - 1, x - 1); t) + (b_{x - 1} + b_x) u^0((x - 1, x); t) \right] \\
 =   \frac{1}{b_{x - 2}} \left[ \frac{d}{dt} \left( q u((x, x - 1); t) + \frac{b_x}{b_{x - 1}} \big(1 - q^2\big) u((x, x); t) \right) - b_{x - 1} (1 + q) u((x - 1, x - 1); t) \right. \\
  \hphantom{=\frac{1}{b_{x - 2}}} \left. \vphantom{\frac{d}{dt}} + (b_{x - 1} + b_x) \left( q u((x, x - 1); t) + \frac{b_x}{b_{x - 1}} \big(1 - q^2\big) u((x, x); t) \right) \right] \\
 =   \frac{1 - q^2}{b_{x - 2}} \left[ \left( b_x + q \frac{b^2_x}{b_{x - 1}} \right) u((x, x); t) + b_x u((x, x - 1); t) - b_{x - 1} u((x - 1, x - 1); t) \right. \\
  \hphantom{=\frac{1 - q^2}{b_{x - 2}}} + \left. \frac{q}{1 - q^2} b_{x - 2} u((x, x - 2); t) \right].
\end{gather*}
\end{Remark}

We note that for any meromorphic function $f(w_1, w_2)$, the integral with parameters~$x_1$,~$x_2$,~$t$
\begin{gather*}
 I_f((x_1, x_2); t) = \frac{1}{b_{x_1} b_{x_2}} \dashint_{C_1} \! dw_1 \dashint_{C_2} \! dw_2 f(w_1, w_2) \left( \prodprime^{x_1}_{k = 1} \frac{b_k}{b_k - w_1} \right) \! \left( \prodprime^{x_2}_{k = 1} \frac{b_k}{b_k - w_2} \right) e^{-w_1 t} e^{-w_2 t}
\end{gather*}
satisf\/ies the dif\/ferential equations \eqref{eq:evolution_eq_2_particles} for $u^0((x_1, x_2); t)$. To see it, we compute
\begin{alignat*}{3}
 & \frac{d}{dt} I_f((x_1, x_2); t) =   I_g((x_1, x_2); t), \qquad && \text{where} \quad  g(w_1, w_2) =   -(w_1 + w_2) f(w_1, w_2),& \\
& I_f((x_1 - 1, x_2); t) =   I_{f_1}((x_1, x_2); t), \qquad && \text{where} \quad f_1(w_1, w_2) =   \frac{b_{x_1} - w_1}{b_{x_1 - 1}} f(w_1, w_2),& \\
& I_f((x_1, x_2 - 1); t) =   I_{f_2}((x_1, x_2); t), \qquad && \text{where} \quad  f_2(w_1, w_2) =   \frac{b_{x_2} - w_2}{b_{x_2 - 1}} f(w_1, w_2).&
\end{alignat*}
We then have that for $I_f((x_1, x_2); t)$ to satisfy \eqref{eq:evolution_eq_2_particles}, it suf\/f\/ices to show $g = b_{x_1 - 1} f_1 + b_{x_2 - 1} f_2 - (b_{x_1} + b_{x_2}) f$, which is straightforward to verify.

In the same way as above, we have that $I_f((x_1, x_2); t)$ satisf\/ies the boundary condition \eqref{eq:BC_2_particles} for $u^0((x_1, x_2); t)$ if and only if for all $x \in \intZ$
\begin{gather} \label{eq:integral_boundary_C}
 \dashint_{C_1} dw_1 \dashint_{C_2} dw_2 f(w_1, w_2) (qw_2 - w_1) \left( \prodprime^x_{k = 1} \frac{b_k}{b_k - w_1} \right) \left( \prodprime^x_{k = 1} \frac{b_k}{b_k - w_2} \right) e^{-w_1 t} e^{-w_2 t} = 0.\!\!\!
\end{gather}
Then inspired by the constructions in \cite{Korhonen-Lee14}, we take $f(w_1, w_2) = F_{y_1, y_2}(w_1, w_2)$, where $(y_1, y_2) \in \intZ^2$, and
\begin{gather*}
 F_{y_1, y_2}(w_1, w_2) = \left( \prodprime^0_{k = y_1} \frac{b_k}{b_k - w_1} \right) \left( \prodprime^0_{k = y_2} \frac{b_k}{b_k - w_2} \right) \\
 \hphantom{F_{y_1, y_2}(w_1, w_2) =}{}
 - \frac{qw_1 - w_2}{qw_2 - w_1} \left( \prodprime^0_{k = y_2} \frac{b_k}{b_k - w_1} \right) \left( \prodprime^0_{k = y_1} \frac{b_k}{b_k - w_2} \right).
\end{gather*}
It is straightforward to check that this $I_{F_{y_1, y_2}}((x_1, x_2); t)$ satisf\/ies identity~\eqref{eq:integral_boundary_C}.

Now we consider the value of $I_{F_{y_1, y_2}}((x_1, x_2); 0)$ for $(x_1, x_2)$ and $(y_1, y_2) \in \Weylchamber^2$. First, in the double integral
\begin{gather} \label{eq:int_1_init_2pt}
 \dashint_{C_1} dw_1 \dashint_{C_2} dw_2 \left( \prodprime^{x_1}_{k = y_1} \frac{b_k}{b_k - w_1} \right) \left( \prodprime^{x_2}_{k = y_2} \frac{b_k}{b_k - w_2} \right),
\end{gather}
if $x_2 < y_2$, then the inner integral with respect to $w_2$ vanishes, since the integrand is holomorphic with respect to $w_2$; if $x_2 > y_2$, then by enlarging the contour $C_2$ to $\{ \lvert w_2 \rvert = M \}$ with $M \to \infty$, we still have that the inner integral vanishes. Next, in the double integral
\begin{gather} \label{eq:int_2_init_2pt}
 \dashint_{C_1} dw_1 \dashint_{C_2} dw_2 \frac{qw_1 - w_2}{qw_2 - w_1} \left( \prodprime^{x_1}_{k = y_2} \frac{b_k}{b_k - w_1} \right) \left( \prodprime^{x_1}_{k = y_1} \frac{b_k}{b_k - w_2} \right),
\end{gather}
if $x_2 < y_2$, then $x_2 < y_1$, and we have that the inner integral with respect to $w_2$ vanishes if $x_2 < y_1$, since the integrand has no pole within $C_2$ with respect to $w_2$; if $x_2 > y_2$, then $x_1 > y_2$, and by enlarging the contour $C_1$ to $\{ \lvert w_1 \rvert = M \}$ with $M \to \infty$, we have that the integral on~$C_1$ with respect to $w_1$ vanishes for $x_1 > y_2$. So the integral~\eqref{eq:int_1_init_2pt} vanishes unless $x_2 = x_1$, and the integral~\eqref{eq:int_2_init_2pt} vanishes unless $x_1 = x_2 = y_1 = y_2 = y$. In the case that $x_2 = y_2$, the integral~\eqref{eq:int_1_init_2pt} is equal to $b_{y_1} b_{y_2}$. Similarly, if $x_1 = x_2 = y_1 = y_2$, the integral~\eqref{eq:int_2_init_2pt} is equal to~$q b^2_y$. Since $I_{F_{y_1, y_2}}(x_1, x_2); 0)$ is $(b_{x_1} b_{x_2})^{-1}$ times the dif\/ference between the two integrals in~\eqref{eq:int_1_init_2pt} and~\eqref{eq:int_2_init_2pt}, we have that
\begin{gather*}
 I_{F_{y_1, y_2}}((x_1, x_2); 0) =
 \begin{cases}
 W((x_1, x_2)) & \text{if $(x_1, x_2) = (y_1, y_2)$}, \\
 0 & \text{if $(x_1, x_2) \in \Weylchamber^2 {\setminus} \{ (y_1, y_2) \}$}.
 \end{cases}
\end{gather*}

We conclude that $u^0((x_1, x_2); t) = I_{F_{y_1, y_2}}((x_1, x_2); t)$ is a solution to \eqref{eq:evolution_eq_2_particles} and \eqref{eq:BC_2_particles}, and by~\eqref{eq:relation_u^0_u_2_particles} it yields $u((x_1, x_2); t)$ that is a solution to the master equation \eqref{eq:master_eq_2_particles} with the initial condition \eqref{eq:init_C_2_particles}. Note that our solution $u((x_1, x_2); t)$ is expressed by a double contour integral formula, and it agrees with the $n = 2$ case of~\eqref{eq:main_result} after a simple change of variables. Thus the $n = 2$ case of Theorem~\ref{thm:main} is proved.

\section[$n$ particle case]{$\boldsymbol{n}$ particle case} \label{sec:n_particle_case}

Now we prove Theorem \ref{thm:main} in the general form. Our strategy is the same as in the proof in the $n = 2$ case in Section~\ref{sec:2_particle_case}.

First we consider functions $u^0(X; t)$, where $X = (x_1, \dotsc, x_n) \in \intZ^n$, such that they are dif\/ferentiable functions in $t$, and satisfy the free particle evolution equation
\begin{gather} \label{eq:general_evolution_u^0}
 \frac{d}{dt} u^0(X; t) = \sum^n_{k = 1} \big( b_{x_k - 1} u^0\big(X^{k, -}; t\big) - b_{x_k} u^0(X; t) \big).
\end{gather}
We also consider the boundary conditions that generalize \eqref{eq:BC_2_particles}, such that for all $k = 1, \dotsc, n - 1$
\begin{gather}
 b_{x - 1} u^0((\dotsc, x_{k - 1}, \underbrace{x - 1, x}_{\text{$x_k$ and $x_{k + 1}$}}, x_{k + 2}, \dotsc); t) = q b_{x - 1} u^0((\dotsc, x_{k - 1}, x, x - 1, x_{k + 2}, \dotsc); t) \nonumber\\
 \qquad{}
+ (1 - q) b_x u^0((\dotsc, x_{k - 1}, x, x, x_{k + 2}, \dotsc); t).\label{eq:general_BC}
\end{gather}
Then let for all $X \in \Weylchamber^n$
\begin{gather} \label{eq:general_u_by_u^0}
 u(X; t) = \frac{1}{W(X)} u^0(X; t).
\end{gather}
We have the following result:
\begin{Lemma} \label{lem:BA_coefficients}
 Suppose $u^0(X; t)$ is a solution to the free particle evolution equation~\eqref{eq:general_evolution_u^0} with boundary condition~\eqref{eq:general_BC}. The functions $u(X; t)$ defined by~\eqref{eq:general_u_by_u^0} satisfy the master equation~\eqref{eq:master_eq_concrete} with $P(X; t) = u(X; t)$.
\end{Lemma}

This lemma is essentially equivalent to the (A) $\Leftrightarrow$ (B) part of \cite[Proposition 2.7]{Borodin-Corwin-Sasamoto14}. For the sake of readability, we give the proof below. We also remark that given functions $u(X; t)$ for $X \in \Weylchamber^n$ such that they satisfy \eqref{eq:master_eq_concrete} with $P(X; t) = u(X; t)$, then we can construct $u^0(X; t)$ inductively for all $X \in \intZ^n$, such that~$u^0(X; t)$ satisfy \eqref{eq:general_evolution_u^0} for all~$X \in \intZ^n$, and are related to~$u(X; t)$ by~\eqref{eq:general_u_by_u^0} if $X \in \Weylchamber^n$. The $n = 2$ case is discussed in detail in Remark~\ref{rmk:inverse_2pt}, and we do not give the proof for the general~$n$ case since it is not related to our main result.

\begin{proof}[Proof of Lemma \ref{lem:BA_coefficients}]
 Functions $u(X; t)$ satisfy \eqref{eq:master_eq_concrete} if and only if for all $X$ expressed in~\eqref{eq:general_X},
 \begin{gather} \label{eq:master_eq_equiv}
 \frac{d}{dt} u^0(X; t) = \sum^m_{i = 1} D_i,
 \end{gather}
 where
 \begin{gather*}
 D_i = \frac{W(X)}{W(X^{K_i, -})} \big(1 - q^{p_{s_i - 1}(X^{K_i, -})}\big) a_{s_i - 1} u^0\big(X^{K_i, -}; t\big) - \big(1 - q^{k_i}\big) a_{s_i} u^0(X; t).
 \end{gather*}
 Below we show that for each $i$,
 \begin{gather} \label{eq:check_Bethe_by_i}
 D_i = \sum^{K_i}_{j = K_{i - 1} + 1} \big( b_{s_i - 1} u^0\big(X^{ j, -}; t\big) - b_{s_i} u^0(X; t) \big).
 \end{gather}
 Plugging \eqref{eq:check_Bethe_by_i} into \eqref{eq:master_eq_equiv}, we have that \eqref{eq:master_eq_equiv} is equivalent to \eqref{eq:general_evolution_u^0}, and then prove the lemma.

 In the case $i = m$ or $i < m$ and $s_{i} - 1 > s_{i + 1}$, we have $p_{s_i - 1}(X^{K_i, -}) = 1$ and $W(X)/W(X^{K_i, -})$ $= (1 - q^{k_i})/(1 - q)$. On the other hand, if $i < m$ and $s_i - 1 = s_{i + 1}$, we have $p_{s_i - 1}(X^{K_i, -}) = k_{i + 1} + 1$ and $W(X)/W(X^{K_i, -}) = (1 - q^{k_i})/(1 - q^{k_{i - 1} + 1})$. Then in both cases, we have
 \begin{gather}
 D_i = (1 - q^{k_i}) \big( a_{s_i - 1} u^0\big(X^{K_i, -}; t\big) - a_{s_i} u^0(X; t) \big) \nonumber\\
 \hphantom{D_i}{} = \frac{1 - q^{k_i}}{1 - q} \big( b_{s_i - 1} u^0\big(X^{K_i, -}; t\big) - b_{s_i} u^0(X; t) \big).\label{eq:simple_comp_i}
 \end{gather}
 If $k_i = 1$, then \eqref{eq:simple_comp_i} gives \eqref{eq:check_Bethe_by_i} directly. Otherwise, we use \eqref{eq:general_BC} recursively with $j = K_i - 1$, $K_i - 2, \dotsc, K_{i - 1} + 2, K_{i - 1} + 1$ to derive
 \begin{gather*}
 b_{s_i - 1} u^0\big(X^{j, -}; t\big) = q^{K_i - j} b_{s_i - 1} u^0\big(X^{K_i, -}; t\big) + \big(1 - q^{K_i - j}\big) b_{s_i} u^0(X; t),
 \end{gather*}
 and then verify \eqref{eq:check_Bethe_by_i}.
\end{proof}

Next, we use the idea of the Bethe ansatz to construct the solution to equations~\eqref{eq:general_evolution_u^0} and~\eqref{eq:general_BC} with initial condition
\begin{gather} \label{eq:init_cond_u^0_general}
 u^0(X; 0) =
 \begin{cases}
 W(Y) & \text{if $X = Y \in \Weylchamber^n$}, \\
 0 & \text{if $X \in \Weylchamber^n {\setminus} \{ Y \}$}.
 \end{cases}
\end{gather}
The existence of the solution is not obvious, and it depends on the completeness of the Bethe ansatz for this model, which is so far only proved in~\cite{Borodin-Corwin-Petrov-Sasamoto15} in the homogeneous rate parameter case. Nevertheless, we imitate the construction in~\cite{Korhonen-Lee14} and guess the formula. Then we simply check that the formula is correct.

\begin{Lemma} \label{lem:Bethe_ansatz_solution}
 Let $Y \in \Weylchamber^n$. The function
 \begin{gather*}
 u^0_Y(X; t) = \sum_{\sigma \in S_n} \Lambda_Y(X; t; \sigma)
 \end{gather*}
 satisfies equation~\eqref{eq:general_evolution_u^0}, boundary condition~\eqref{eq:general_BC} and initial condition~\eqref{eq:init_cond_u^0_general}.
\end{Lemma}

\begin{proof}
\textit{Free particle evolution equations \eqref{eq:general_evolution_u^0}.}
 It is straightforward to check that for all $\sigma$, $\Lambda_Y(X; t; \sigma)$ satisfy \eqref{eq:general_evolution_u^0}, so $u^0_Y(X; t)$ also satisf\/ies \eqref{eq:general_evolution_u^0}.

\textit{Boundary conditions \eqref{eq:general_BC}.}
 To check that $u^0_Y(X; t)$ satisf\/ies \eqref{eq:general_BC}, we f\/ix $k \in \{ 1, \dotsc, n - 1 \}$, and denote
 \begin{gather*}
 X(a, b) = (x_1, \dotsc, x_{k - 1}, \!\!\underbrace{a, b}_{x_k \text{ and } x_{k + 1}}\!\! , x_{k + 2}, \dotsc, x_n).
 \end{gather*}
 We divide $S_n$ into $n!/2$ pairs, such that permutations $\mu$ and $\nu$ are in a pair if and only if $\mu = \nu \circ (k, k + 1)$ and $\nu = \mu \circ (k, k + 1)$, that is, $\mu(i) = \nu(i)$ if $i \neq k, k + 1$, and $\mu(k) = \nu(k + 1)$, $\mu(k + 1) = \nu(k)$. Below we assume that $\mu$ and $\nu$ are in a pair such that $\mu(k) = \nu(k + 1) = \alpha < \beta = \mu(k + 1) = \nu(k)$. It is not hard to check that
 \begin{gather*}
 A_{\nu} = A_{\mu} S_{(\beta, \alpha)}.
 \end{gather*}
 For notational simplicity, we denote
 \begin{gather*}
 F(w_1, \dotsc, w_n) =  \frac{1}{b_x} A_{\mu} \prod^n_{l = 1} e^{-w_l t} \prod_{\substack{j = 1, \dotsc, n \\ j \neq k, k + 1}} \frac{-1}{b_{x_j}} \prodprime^{x_j}_{l_j = y_{\mu(j)}} \frac{b_{l_j}}{b_{l_j} - w_{\mu(j)}}, \\
 G(w_1, \dotsc, w_n) =  F(w_1, \dotsc, w_n) \prod^x_{i = y_{\alpha}} \prodprime^x_{j = y_{\beta}} \frac{b_i b_j}{(b_i - w_{\alpha})(b_j - w_{\beta})}.
 \end{gather*}
 Then we have
 \begin{gather}
   b_{x - 1} \left( \Lambda_Y(X(x - 1, x); t; \mu) + \Lambda_Y(X(x - 1, x); t; \nu) \right)\nonumber \\
   \qquad{}
 =   \dashint_{C_1} dw_1 \dotsi \dashint_{C_n} dw_n F(w_1, \dotsc, w_n) \left( \prodprime^{x - 1}_{l_k = y_{\alpha}} \frac{b_{l_k}}{b_{l_k} - w_{\alpha}} \prodprime^x_{l_{k + 1} = y_{\beta}} \frac{b_{l_{k + 1}}}{b_{l_{k + 1}} - w_{\beta}} \right. \nonumber\\
\left. \qquad\quad{} - \frac{qw_{\beta} - w_{\alpha}}{qw_{\alpha} - w_{\beta}} \prodprime^{x - 1}_{l_k = y_{\beta}} \frac{b_{l_k}}{b_{l_k} - w_{\beta}} \prodprime^x_{l_{k + 1} = y_{\alpha}} \frac{b_{l_{k + 1}}}{b_{l_{k + 1}} - w_{\alpha}} \right) \nonumber\\
 \qquad {}=   \dashint_{C_1} dw_1 \dotsi \dashint_{C_n} dw_n G(w_1, \dotsc, w_n) \left( \frac{b_x - w_{\alpha}}{b_x} - \frac{qw_{\beta} - w_{\alpha}}{qw_{\alpha} - w_{\beta}} \frac{b_x - w_{\beta}}{b_x} \right).
\label{eq:boundary_integral_x-1_x}
 \end{gather}
 Similarly,
 \begin{gather}
 b_x \left( \Lambda_Y(X(x, x); t; \mu) + \Lambda_Y(X(x, x); t; \nu) \right) \nonumber\\
 \qquad{}{} = \dashint_{C_1} dw_1 \dotsi \dashint_{C_n} dw_n G(w_1, \dotsc, w_n) \left( 1 - \frac{qw_{\beta} - w_{\alpha}}{qw_{\alpha} - w_{\beta}} \right), \label{eq:boundary_integral_x_x}
 \\
 b_{x - 1} \left( \Lambda_Y(X(x, x - 1); t; \mu) + \Lambda_Y(X(x, x - 1); t; \nu) \right) \nonumber\\
 \qquad{} = \dashint_{C_1} dw_1 \dotsi \dashint_{C_n} dw_n G(w_1, \dotsc, w_n)
  \left( \frac{b_x - w_{\beta}}{b_x} - \frac{qw_{\beta} - w_{\alpha}}{qw_{\alpha} - w_{\beta}} \frac{b_x - w_{\alpha}}{b_x} \right).\label{eq:boundary_integral_x_x-1}
 \end{gather}
 With the integral expressions \eqref{eq:boundary_integral_x-1_x}, \eqref{eq:boundary_integral_x_x} and \eqref{eq:boundary_integral_x_x-1}, and the algebraic identity
 \begin{gather*}
 q \left( \frac{b_x - w_{\beta}}{b_x} - \frac{qw_{\beta} - w_{\alpha}}{qw_{\alpha} - w_{\beta}} \frac{b_x - w_{\alpha}}{b_x} \right) + (1 - q) \left( 1 - \frac{qw_{\beta} - w_{\alpha}}{qw_{\alpha} - w_{\beta}} \right) \\
 \qquad{} = \frac{b_x - w_{\alpha}}{b_x} - \frac{qw_{\beta} - w_{\alpha}}{qw_{\alpha} - w_{\beta}} \frac{b_x - w_{\beta}}{b_x},
 \end{gather*}
 it is clear that
 \begin{gather*}
 b_{x - 1} \sum_{\sigma \in \{ \mu, \nu \}} \Lambda_Y(X(x - 1, x); t; \sigma)   \\
 \qquad{} = q b_{x - 1} \sum_{\sigma \in \{ \mu, \nu \}} \Lambda_Y(X(x, x - 1); t; \sigma) + (1 - q) b_x \sum_{\sigma \in \{ \mu, \nu \}} \Lambda_Y(X(x, x); t; \sigma).
 \end{gather*}
 Since the identity holds for all the $n!/2$ pairs in $S_n$, we sum up them together, and conclude that $u^0_Y(X; t)$ satisf\/ies \eqref{eq:general_BC}.

\textit{Initial condition~\eqref{eq:init_cond_u^0_general}.}
 We check the initial condition inductively. The $n = 1$ case can be done by direct computation, and if we assume that~\eqref{eq:init_cond_u^0_general} holds when the particle number is $n - 1$, we show that it also holds when the particle number is~$n$.

 For any $k = 1, \dotsc, n$, we def\/ine two subsets of $S_n$:
 \begin{gather}
 S_n(k) = \{ \sigma \in S_n \,|\, \sigma(1) = k \}, \qquad \text{and} \nonumber\\
  S^{-1}_n(k) = \{ \sigma \in S_n \,|\, \sigma(k) = 1 \} = \big\{ \sigma \,|\, \sigma^{-1} \in S_n(k) \big\}.\label{eq:defn_S_n(k)}
 \end{gather}
 Then we def\/ine two one-to-one correspondences $\varphi_k\colon S_n(k) \to S_{n - 1}$ and $\psi_k\colon S^{-1}_n(k) \to S_{n - 1}$, such that
for all $i = 1, \dotsc, n - 1$
 \begin{gather}
 \varphi_k(\sigma) =   \tau \quad \text{if\/f} \quad \sigma(i + 1) =
 \begin{cases}
 \tau(i) & \text{if $\tau(i) < k$}, \\
 \tau(i) + 1 & \text{if $\tau(i) \geq k$},
 \end{cases} \nonumber\\
 \psi_k(\sigma) =   \lambda \quad \text{if\/f} \quad \sigma^{-1}(i + 1) =
 \begin{cases}
 \lambda^{-1}(i) & \text{if $\lambda^{-1}(i) < k$}, \\
 \lambda^{-1}(i) + 1 & \text{if $\lambda^{-1}(i) \geq k$}.
 \end{cases}
\label{eq:defn_varphi_k_and_psi_k}
 \end{gather}
 As an example, $\sigma = (123) = \left(
 \begin{smallmatrix}
 1 & 2 & 3 \\
 2 & 3 & 1
 \end{smallmatrix}
 \right) \in S_3$ is in $S_3(2)$ and $S_3(3)$, with $\varphi_2(\sigma) = (12) \in S_2$ and $\psi_3(\sigma) = \id \in S_2$.

 We have for any $\sigma \in S_n$,
\begin{gather}
 A_{\sigma}(w_1, \dotsc, w_n) \nonumber\\
\qquad{} =
 \begin{cases}
 \displaystyle \prod^{k - 1}_{j = 1} \frac{q w_k - w_j}{w_k - q w_j} A_{\tau}(w_1, \dotsc, w_{ k - 1}, w_{k + 1}, w_n) & \text{if $\sigma(1) = k$ and $\varphi_k(\sigma) = \tau$}, \\
 \displaystyle \prod^{k - 1}_{j = 1} \frac{q w_{\sigma(j)} - w_1}{w_{\sigma(j)} - q w_1} A_{\lambda}(w_2, \dotsc, w_n) & \text{if $\sigma(k) = 1$ and $\psi_k(\sigma) = \lambda$}.
 \end{cases}\label{eq:A_sigma_in_A_tau_A_lambda}
 \end{gather}

 We write
 \begin{gather*}
 u^0_Y(X; 0) = \sum^n_{k = 1} I^{(k)}_Y(X), \qquad \text{where} \quad I^{(k)}_Y(X) = \sum_{\sigma \in S_n(k)} \Lambda_Y(X; 0; \sigma),
 \end{gather*}
 and compute $I^{(k)}_Y(X)$ for $k = 1, \dotsc, n$, under the assumption that~\eqref{eq:init_cond_u^0_general} holds when $n$ is replaced by $n - 1$.

 In the simplest case $k = 1$, it is straightforward. Note that if $\sigma \in S_n(1)$, by~\eqref{eq:A_sigma_in_A_tau_A_lambda} $A_{\sigma}(w_1, \dotsc, w_n)$ does not depend on $w_1$, and we have ($(x_1)$ is represented by~$x_1$ as in Section~\ref{sec:1pt})
 \begin{gather*}
  \left( (-1)^n \prod^n_{i = 1} b_{x_i} \right) I^{(1)}_Y(X) \\
 =   \dashint_{C_1} dw_1 \prodprime^{x_1}_{l_1 = y_1} \frac{b_{l_1}}{b_{l_1} - w_1} \sum_{\sigma \in S_n(1)} \dashint_{C_2} dw_2 \dotsi \dashint_{C_n} dw_n A_{\sigma}(w_1, \dotsc, w_n) \prod^n_{j = 2} \prodprime^{x_j}_{l_j = y_{\sigma(j)}} \frac{b_{l_j}}{b_{l_j} - w_{\sigma(j)}} \\
 =   \dashint_{C_1} dw_1 \prodprime^{x_1}_{l_1 = y_1} \frac{b_{l_1}}{b_{l_1} - w_1} \sum_{\tau \in S_{n - 1}} \dashint_{C_2} dw_2 \dotsi \dashint_{C_n} dw_n A_{\tau}(w_2, \dotsc, w_n) \prod^{n - 1}_{j = 1} \prodprime^{x_{j + 1}}_{l_j = y_{\tau(j) + 1}} \frac{b_{l_j}}{b_{l_j} - w_{\tau(j) + 1}} \\
 =  \left( (-1)^n \prod^n_{i = 1} b_{x_i} \right) u^0_{y_1}(x_1; 0) u^0_{(y_2, \dotsc, y_n)}((x_2, \dotsc, x_n); 0).
 \end{gather*}
 So by the inductive assumption,
 \begin{gather*}
 I^{(1)}_Y(X) =
 \begin{cases}
 W((x_2, \dotsc, x_n)) & \text{if $X = Y$}, \\
 0 & \text{if $X \in \Weylchamber^n {\setminus} \{ Y \}$}.
 \end{cases}
 \end{gather*}

 In the general case $k > 1$, for all $\sigma \in S_n(k)$, if $S_{n - 1} \ni \tau = \varphi_k(\sigma)$, we have
 \begin{gather}
 \left( (-1)^n \prod^n_{i = 1} b_{x_i} \right) \Lambda_Y(X; 0; \sigma) \nonumber\\
\qquad{} = \dashint_{C_1} dw_1 \dotsi \dashint_{C_{k - 1}} dw_{k - 1} \dashint_{C_{k + 1}} dw_{k + 1} \dotsi \dashint_{C_n} dw_n A_{\tau}(w_1, \dotsc, w_{k - 1}, w_{k + 1}, \dotsc, w_n) \nonumber\\
 \qquad\quad{} \times \prod^n_{j = 2} \prodprime^{x_j}_{l_j = y_{\sigma(j)}} \frac{b_{l_j}}{b_{l_j} - w_{\sigma(j)}} \left( \dashint_{C_k} \prod^{k - 1}_{i = 1} \frac{q w_k - w_j}{q w_j - w_k} \prodprime^{x_1}_{l_1 = y_k} \frac{b_{l_1}}{b_{l_1} - w_k} dw_k \right).\label{eq:Lambda_YX;0;sigma_first}
 \end{gather}
 The inner integral vanishes if $x_1 > y_k$, so $\Lambda_Y(X; 0; \sigma) = 0$ if $x_1 > y_k$. To see it, we deform the contour $C_k$ into a circle centred at $0$ with radius $R_k \to \infty$ while keeping the other $C_j$ ($j \neq k$) in place, and use the estimate that the integrand of the inner integral is $\bigO\big(w^{y_k - x_1 - 1}_k\big)$ as $w_k \to \infty$. Note that this deformation does not change the integral on the right-hand side of~\eqref{eq:Lambda_YX;0;sigma_first}, since the only requirement for $C_k$ is that it is big enough to enclose all possible poles~$b_{l_1}$.

 Next, we assume that $\sigma \in S^{-1}_n(m)$, and $\psi_m(\sigma) = \lambda$. Then we have
 \begin{gather}
 \left( (-1)^n \prod^n_{i = 1} b_{x_i} \right) \Lambda_Y(X; 0; \sigma) = \dashint_{C_2} dw_2 \dotsi \dashint_{C_n} dw_n A_{\lambda}(w_2, \dotsc, w_n) \prod^n_{j = 2} \prodprime^{x_{\sigma^{-1}(j)}}_{l_j = y_j} \frac{b_{l_j}}{b_{l_j} - w_j}\nonumber \\
 \hphantom{\left( (-1)^n \prod^n_{i = 1} b_{x_i} \right) \Lambda_Y(X; 0; \sigma) =}{}
 \times \left( \dashint_{C_1} \prod^{k - 1}_{j = 1} \frac{q w_{\sigma(j)} - w_1}{w_{\sigma(j)} - q w_1} \prodprime^{x_m}_{l_1 = y_1} \frac{b_{l_1}}{b_{l_1} - w_1} dw_1 \right).\label{eq:Lambda_YX;0;sigma_second}
 \end{gather}
 If $x_m < y_1$, then the inner integral in~\eqref{eq:Lambda_YX;0;sigma_second} vanishes, for its integrand has no pole within~$C_1$. So~$\Lambda_Y(X; 0; \sigma) = 0$ if~$x_m < y_1$.

 Since $x_1 \geq \dotsb \geq x_m$ and $y_1 \geq \dotsb \geq y_k$, by argument above, we see that $\Lambda_Y(X; 0; \sigma) = 0$ unless $x_1 = \dotsb = x_m = y_1 = \dotsb = y_k$, and in particular, $x_1 = y_1 = \dotsb = y_k$.

 When $x_1 = y_1 = \dotsb = y_k$, by deforming the contour $C_k$ into a big circle centred at $0$ with radius $R_k \to \infty$, the inner integral in \eqref{eq:Lambda_YX;0;sigma_first} can be explicitly computed as
 \begin{gather} \label{eq:inner_integral}
 \dashint_{C_k} \prod^{k - 1}_{i = 1} \frac{q w_k - w_j}{w_k - q w_j} \frac{b_{x_1}}{b_{x_1} - w_k} dw_k = -q^{k - 1} b_{x_1},
 \end{gather}
 since its integrand is $-q^{k - 1} b_{x_1} w^{-1}_k \big(1 + \bigO\big(w^{-1}_k\big)\big)$. Plugging in~\eqref{eq:inner_integral} into~\eqref{eq:Lambda_YX;0;sigma_first}, we have
 \begin{gather}
   \left( (-1)^n \prod^n_{i = 1} b_{x_i} \right) \Lambda_Y(X; 0; \sigma)
 =  \dashint_{C_1} dw_1 \dotsi \dashint_{C_{k - 1}} dw_{k - 1} \dashint_{C_{k + 1}} dw_{k + 1} \dotsi \nonumber\\
 \qquad\quad{}\times \dashint_{C_n} dw_n A_{\tau}(w_1, \dotsc, w_{k - 1}, w_{k + 1}, \dotsc, w_n)  \prod^n_{j = 2} \prodprime^{x_j}_{l_j = y_{\sigma(j)}} \frac{b_{l_j}}{b_{l_j} - w_{\sigma(j)}} q^{k - 1} b_{x_1} \nonumber\\
\qquad{}
 =  q^{k - 1} \left( (-1)^n \prod^n_{i = 1} b_{x_i} \right) \Lambda_{(y_2, \dotsc, y_n)}((x_2, \dotsc, x_n); 0; \tau),\label{eq:evanluation_x_1=y_k}
  \end{gather}
 where $\tau = \varphi_k(\sigma)$. Thus summing up with $\sigma \in S_n(k)$, we have that $\tau$ runs over $S_{n - 1}$ and
 \begin{gather*} 
 I^{(k)}_Y(X) = q^{k - 1} u^0_{(y_2, \dotsc, y_n)}((x_2, \dotsc, x_n); 0).
 \end{gather*}
 If $X \neq Y$, then the assumption $x_1 = y_1 = \dotsb = y_k$ implies $(x_2, \dotsc, x_n) \neq (y_2, \dotsc, y_n)$, and by inductive assumption have that $I^{(k)}_Y(X) = q^{k - 1} \cdot 0 = 0$. Otherwise, $X = Y$, and then the assumption $x_1 = y_1 = \dotsb = y_k$ holds only if $k \leq k_1$, if $X$ is expressed in the form of~\eqref{eq:general_X}.

 The result derived above can be summarized as follows: If $X$ is expressed by~\eqref{eq:general_X}, then by the inductive assumption,
 \begin{gather*}
 I^{(k)}_X(X) =
 \begin{cases}
 0 & \text{if $X \in \Weylchamber^n {\setminus} \{ Y \}$}, \\
 0 & \text{if $X = Y$ and $k > k_1$}, \\
 q^{k - 1} W((x_2, \dotsc, x_n)) & \text{if $X = Y$ and $k \leq k_1$},
 \end{cases}
 \end{gather*}
 and hence $u^0_Y(X) = 0$ if $X \in \Weylchamber^n {\setminus} \{ Y \}$ and $u^0_X(X) = \big(1 + q + \dotsb + q^{k_1 - 1}\big) W((x_2, \dotsc, x_n)) = W(X)$, and we prove~\eqref{eq:init_cond_u^0_general}.
\end{proof}

Combining Lemmas \ref{lem:Bethe_ansatz_solution} and \ref{lem:BA_coefficients}, we prove Theorem \ref{thm:main} in the general $n$ particle case.

\section{Proof of Corollary \ref{cor:step_init}} \label{sec:proof_of_cor}

By Theorem \ref{thm:main}, we have
\begin{gather*}
  P_{(0, \dotsc, 0)}(X; t)
 =   \frac{1}{W(X)} \left( \prod^n_{k = 1} \frac{-1}{b_{x_k}} \right) \sum_{\sigma \in S_n} \dashint_{C_1} dw_1 \dotsi \\
 \hphantom{P_{(0, \dotsc, 0)}(X; t) =}{} \times \dashint_{C_n} dw_n A_{\sigma}(w_1, \dotsc, w_n) \prod^n_{j = 1} \left[ \prodprime^{x_j}_{k = 0} \left( \frac{b_k}{b_k - w_{\sigma(j)}} \right) e^{-w_{\sigma(j)} t} \right] \\
 \hphantom{P_{(0, \dotsc, 0)}(X; t)}{}
 =   \frac{1}{W(X)} \left( \prod^n_{k = 1} \frac{-1}{b_{x_k}} \right) \sum_{\sigma \in S_n} \dashint_{C_1} dw_1 \dotsi
\\
\hphantom{P_{(0, \dotsc, 0)}(X; t) =}{} \times
\dashint_{C_n} dw_n A_{\sigma}(w_{\sigma^{-1}(1)}, \dotsc, w_{\sigma^{-1}(n)}) \prod^n_{j = 1} \left[ \prodprime^{x_j}_{k = 0} \left( \frac{b_k}{b_k - w_j} \right) e^{-w_j t} \right],
\end{gather*}
where we utilize that the contours $C_j$ can be taken as identical. Hence it is clear that if we can prove
\begin{gather} \label{eq:id_to_prove_cor}
 \sum_{\sigma \in S_n} A_{\sigma}(w_{\sigma^{-1}(1)}, \dotsc, w_{\sigma^{-1}(n)}) = [n]_q! B(w_1, \dotsc, w_n),
\end{gather}
then \eqref{eq:step_init_trans_prob} is a direct consequence of \eqref{eq:main_result}. We prove \eqref{eq:id_to_prove_cor} by induction. For $n = 1, 2$, \eqref{eq:id_to_prove_cor} holds obviously. Suppose \eqref{eq:id_to_prove_cor} holds with $n$ replaced by $n - 1$. Recall the notations $S^{-1}_n(k)$ def\/ined in \eqref{eq:defn_S_n(k)} and $\psi_k$ def\/ined in \eqref{eq:defn_varphi_k_and_psi_k} for $k = 1, \dotsc, n$, and let $v^{(k)}_i = w_i$ for $i = 1, \dotsc, k - 1$ and $v^{(k)}_i = w_{i + 1}$ for $i = k, \dotsc, n - 1$. We have
\begin{align*}
 \sum_{\sigma \in S^{-1}_n(k)} A_{\sigma}(w_{\sigma^{-1}(1)}, \dotsc, w_{\sigma^{-1}(n)}) & =   \prod^{k - 1}_{j = 1} \frac{q w_j - w_k}{w_j - q w_k} \sum_{\sigma \in S^{-1}_n(k)} A_{\psi_k(\sigma)}(w_{\sigma^{-1}(2)}, \dotsc, w_{\sigma^{-1}(n)}) \nonumber\\
 &=   \prod^{k - 1}_{j = 1} \frac{q w_j - w_k}{w_j - q w_k} \sum_{\lambda \in S_{n - 1}} A_{\lambda}\big(v^{(k)}_{\lambda^{-1}(1)}, \dotsc, v^{(k)}_{\lambda^{-1}(n - 1)}\big) \nonumber\\
&  =   [n - 1]_q! \prod^{k - 1}_{j = 1} \frac{q w_j - w_k}{w_j - q w_k} B\big(v^{(k)}_1, \dotsc, v^{(k)}_{n - 1}\big),
\end{align*}
and then
\begin{gather*}
  \sum_{\sigma \in S_n} A_{\sigma}(w_{\sigma^{-1}(1)}, \dotsc, w_{\sigma^{-1}(n)})
 = \sum^n_{k = 1} \left( \sum_{\sigma \in S^{-1}_n(k)} A_{\sigma}(w_{\sigma^{-1}(1)}, \dotsc, w_{\sigma^{-1}(n)}) \right) \\
\qquad{} =    [n - 1]_q! \sum^n_{k = 1} \prod^{k - 1}_{j = 1} \frac{q w_j - w_k}{w_j - q w_k} B\big(v^{(k)}_1, \dotsc, v^{(k)}_{n - 1}\big) \\
 \qquad{} =   [n - 1]_q! \sum^n_{k = 1} \prod^{k - 1}_{j = 1} \frac{q w_j - w_k}{w_j - q w_k} \left( \prod^{k - 1}_{j = 1} \frac{q w_k - w_j}{w_k - w_j} \prod^n_{l = k + 1} \frac{q w_l - w_k}{w_l - w_k} B(w_1, \dotsc, w_n) \right) \\
 \qquad{} =   [n - 1]_q! C(w_1, \dotsc, w_n) B(w_1, \dotsc, w_n),
\end{gather*}
where
\begin{gather*} 
 C(w_1, \dotsc, w_n) = \sum^n_{k = 1} \prod_{j = 1, \dotsc, n, \ j \neq k} \frac{q w_j - w_k}{w_j - w_k}.
\end{gather*}
To prove \eqref{eq:id_to_prove_cor}, it suf\/f\/ices to show
\begin{gather} \label{eq:tech_id_for_cor}
 C(w_1, \dotsc, w_n) = \frac{[n]_q!}{[n - 1]_q!} = \frac{1 - q^n}{1 - q}.
\end{gather}
Denoting the function $h(z) = z^{-1} \prod\limits^n_{k = 1} (z - q w_k)/(z - w_k)$ that is meromorphic in $z$, we have
\begin{gather*}
 (1 - q) C(w_1, \dotsc, w_n) = \sum^n_{k = 1} \Res_{z = w_k} h(z) = -\Res_{z = \infty} h(z) - \Res_{z = 0} h(z) = 1 - q^n.
\end{gather*}
Thus we verify \eqref{eq:tech_id_for_cor} and f\/inish the proof of the corollary. (Inspired by an anonymous referee's suggestion, the authors f\/ind that \eqref{eq:tech_id_for_cor} can be proved algebraically by taking $z = 0$ in \cite[the f\/irst formula on p.~210]{Macdonald95}.)

\subsection*{Acknowledgements}

The authors are indebted to Ivan Corwin for numerous comments, especially on the relation to the spectral theory and on the proof of Corollary \ref{cor:step_init}. We also thank Eunghyun Lee for valuable suggestions and comments. At last, we thank the anonymous referees for their comments, suggestions and corrections, especially on a sign error in the manuscript. The f\/irst named author is supported partially by the startup grant R-146-000-164-133.

\pdfbookmark[1]{References}{ref}
\LastPageEnding

\end{document}